\documentclass[12pt]{article}

\usepackage{url}
\usepackage{mathtools}
\usepackage{amssymb}
\usepackage{amsthm}
\usepackage{empheq}
\usepackage{latexsym}
\usepackage{enumitem}
\usepackage{eurosym}
\usepackage{dsfont}
\usepackage{appendix}
\usepackage{color} 
\usepackage[unicode]{hyperref}
\usepackage{frcursive}
\usepackage[utf8]{inputenc}
\usepackage[T1]{fontenc}
\usepackage{geometry}
\usepackage{multirow}
\usepackage[colorinlistoftodos]{todonotes}
\usepackage{lmodern}
\usepackage{anyfontsize}
\usepackage{stmaryrd}
\usepackage{natbib}
\usepackage{cleveref}

\usepackage{amsbsy}

\usepackage{cancel}

\setcitestyle{numbers,open={[},close={]}}

\definecolor{red}{rgb}{0.7,0.15,0.15}
\definecolor{green}{rgb}{0,0.5,0}
\definecolor{blue}{rgb}{0,0,0.7}
\hypersetup{colorlinks, linkcolor={red},citecolor={green}, urlcolor={blue}}
			
\makeatletter \@addtoreset{equation}{section}

\newtheorem{theorem}{Theorem}[section]
\newtheorem{assumption}[theorem]{Assumption}
\newtheorem{corollary}[theorem]{Corollary}

\newtheorem{lemma}[theorem]{Lemma}
\newtheorem{proposition}[theorem]{Proposition}

\def\no{\noindent}
\def\beq{\begin{eqnarray}}
\def\eeq{\end{eqnarray}}
\def\be*{\begin{eqnarray*}}
\def\ee*{\end{eqnarray*}}

\def \E{\mathbb{E}}

\def \N{\mathbb{N}}

\def \P{\mathbb{P}}
\def \Q{\mathbb{Q}}
\def \R{\mathbb{R}}

\def\Fc{{\cal F}}

\def\Hc{{\cal H}}

\def\Lc{{\cal L}}

\def\x{\times}

\def\0{\mathbf{0}}

\def\normeL2#1{\left\|{#1}\right\|_{L^2}}

\def \Frac{\displaystyle\frac}

\def \Liminf{\displaystyle\liminf}

\def \1{\mathds{1}}
\def \d{{\rm d}}
\def \i{{\rm i}}

\def \Re{\mathrm{Re}}

\usepackage{titletoc}

\dottedcontents{section}[1.3em]{}{1.3em}{1pc}
\dottedcontents{subsection}[2.8em]{}{2em}{1pc}

\def\restrict#1{\raise-.5ex\hbox{\ensuremath|}_{#1}}

\DeclareUnicodeCharacter{014D}{\=o}
 \title{Global Energy Solutions for Normalized Linear Stochastic Schr\"odinger Equations}

\author{
 Théo Hérouard\thanks{CMAP CNRS, Ecole polytechnique, I.P. Paris, 91128 Palaiseau,
        theo.herouard@polytechnique.edu.  }
    }
             \date{\today}

\begin{document}

\maketitle
 
\begin{abstract}
In this paper, we study nonlinear Schr\"odinger equations arising in the study of Markovian open quantum systems. These equations are derived from the normalization of linear stochastic Schr\"odinger equations. First, we prove the global well-posedness of the linear equation in the energy space, which is less regular than the usual space for this equation. Then, we show the indistinguishability of solutions to the nonlinear SPDE under a relaxation of the usual monotonicity assumption. Finally, by proving the pathwise uniqueness and normalizing the solutions to the linear equation, we are able to prove the well-posedness of solutions to the required class of nonlinear stochastic equations.

\tableofcontents
\end{abstract}
%\newpage
\vspace{3mm}
\no{\bf Keywords.} Stochastic partial differential equations, Open quantum systems, Existence and uniqueness of solutions

\vspace{3mm}
\no{\bf MSC2020.} 60H15, 35Q40, 60H30

\section{Introduction}

In this paper, we present a mathematical analysis of the nonlinear Schr\"odinger equation that describes open quantum systems under the Born-Markov approximation. More precisely, we study a stochastic equation on a separable Hilbert space $H$ of the type
\begin{equation}
    \label{eq:SNLSintro}
    \begin{aligned}
    \d \psi(t) &= -\i \Hc \psi(t) \d t + \sum_{k=1}^{\infty}\left(L_k \psi(t) -   \langle L_k \rangle_{\psi(t)} \psi(t)\right) \d W^k(t)\\
    &- \frac{1}{2}\sum_{k=1}^{\infty} \left(L_k^* L_k  - 2  \langle L_k \rangle_{\psi(t)} L_k  +  \langle L_k \rangle_{\psi(t)}^2 \right) \psi(t) \d t,
    \end{aligned}
\end{equation}
where $\Hc$ is an unbounded self-adjoint operator called the Hamiltonian, $(L_k)_{k\geq 1}$ are unbounded operators and $(W^k)_{k \geq 1}$ are independent Brownian motions. The term $\langle L_k \rangle_{\psi}$ is the average value of $L_k$ in state $\psi$ where the definition is given by \eqref{eq:defaverage}.

In quantum mechanics, a system with states in the (potentially infinite dimensional) Hilbert space $H$ evolves under the Hamiltonian $\Hc$ according to the Schr\"odinger equation. If one wants to perform a continuous measurement of the system, an external field in interaction with the system, where the measurement is done, needs to be introduced. Various approaches exist to study continuous measurement. For an in-depth description of such equations in a finite dimensional setting, we refer to \cite{barchielli2009quantum}. In one such approach, first developed by \cite{belavkin1992quantum}, the theory of classical stochastic equations was used to study quantum filtering (for the derivation from quantum probability, see for instance \cite{bouten2007introduction}). For continuous measurement, the trajectories of the system can also be formally obtained as a limit from a discrete time detection scheme \cite{wiseman1996quantum}. These trajectories are then described by a linear stochastic Schr\"odinger equations, which preserves the mass (i.e. the norm in $H$) only in expectation. Therefore, if one wants to obtain directly a post-measurement state, the equation needs to be normalized. The normalized state then evolves according to the nonlinear equation \eqref{eq:SNLSintro}. This is the reason why we decided to call such equations normalized linear stochastic Schr\"odinger equations instead of nonlinear Schr\"odinger equations since the nonlinearity does not come from the Hamiltonian part, as in the classical NLS or stochastic NLS (as in \cite{debouard2003snls} for instance).

Moreover, a classical model in open quantum systems is to consider a small quantum system, with state space $H$, governed by the Hamiltonian $\Hc$, in interaction with a reservoir. Then, under the Born-Markov approximation, the evolution of the density operators is given by the so-called Lindblad equation
\begin{equation}
    \label{eq:lindblad}
    \d \rho_t = - \i [\Hc, \rho_t] \d t + \sum_{k=1}^{\infty}\left( L_k \rho_t L_k^* - \Frac{1}{2}\{L_k^* L_k, \rho_t\} \right)\d t,
\end{equation}
where the family of operators $(L_k)_{k \geq 1}$ models interaction with the environment, $[\cdot, \cdot]$ is the commutator and $\{\cdot, \cdot\}$ the anti-commutator. An interesting property is that, if $\rho_0$ is a pure state, whose form is given by $(\psi_0, \cdot)_H \psi_0$, then for $t >0$ the evolution of the system $\rho_t$ is given by $\E [(\psi(t), \cdot)_H \psi(t)]$ where $\psi$ is a solution of \eqref{eq:SNLSintro} with $\psi(0) = \psi_0$. In particular, since equation \eqref{eq:SNLSintro} is suitable to do simulation, one can use a Monte Carlo approach to perform numerical simulations of the Lindblad equation (e.g \cite{le2015adaptive}). Moreover, the close connection between equations \eqref{eq:SNLSintro} and \eqref{eq:lindblad} can be used to derive regularity properties on the Lindblad equation \eqref{eq:lindblad} from the stochastic equation \eqref{eq:SNLSintro}, as in \cite{zbMATH06216094}.

These equations appear in a large number of physical situations. The most frequent (and most studied) framework is when the Hamiltonian is the harmonic oscillator and the coupling operators are polynomials of the creation and annihilation operators. The Gaussian dynamical semigroups \cite{zbMATH05163346}, the birth and death process \cite{gleyzes2007quantum}, the damped harmonic oscillator \cite{isar1993density, salama1993explicit} and the continuous position measurement \cite{jacobs2006straightforward} all fall within this framework. Another interesting model is that of quantum decoherence, in which the operator $L$ is given by $\sqrt{\gamma} H$ \cite{prezhdo1999mean}. Moreover, if instead for all integer $k$, $L_k$ is chosen equal to $-\i R_k$, with $R_k$ self-adjoint, then equation \eqref{eq:SNLSintro} is used to describe the random unitary evolution of a system. Indeed, if for simplicity there is only one operator $L$, equation \eqref{eq:SNLSintro} can be rewritten formally as
\begin{equation*}
    \i \d u = \left(\Hc + \Dot{W}(t)R \right) u \d t,
\end{equation*}
where $W$ is a Brownian motion, see for instance \cite{semina2014stochastic}.

When $(L_k)_{k \geq 1}$ are unbounded operators (which corresponds to a large proportion of situations), the well-posedness of "good" solutions to \eqref{eq:lindblad} is non-trivial and needs some technical assumptions. We will work with the framework of \cite{chebotarev1998sufficient} where the authors gave a non-explosion condition by introducing a reference operator $C$ (see also \cite{zbMATH00878015}). This operator allows simultaneous control of the Hamiltonian part $\Hc$ and the dissipative part $L^*L$ of the system. In the absence of interaction with the environment, a good reference operator is the Hamiltonian $\Hc$, and the natural space to build solution is then the energy space, i.e. $D(\sqrt{\Hc})$  the domain of the quadratic form associated to the operator $\Hc$.

Following the ideas coming from continuous measurement, we first study a linear stochastic Schr\"odinger equation with a correction which preserves the square norm in expectation:
\begin{equation*}
    \d u(t) = \left(- \i \Hc  - \Frac{1}{2}\sum_{k=1}^{\infty} L^*_k L_k \right)u(t) \d t + \sum_{k=1}^{\infty} L_k u(t) \d B^k(t),
\end{equation*}
From this expression, it is natural to define the effective Hamiltonian $G$ associated to this system by the following expression
\begin{equation}
    \label{eq:defG}
    G := -\i \Hc - \Frac{1}{2} \sum_{k=1}^{\infty} L_k^* L_k.
\end{equation}
Then, the renormalized process $(u(t)/ |u(t)|_H)_{t \geq 0}$ is a solution to \eqref{eq:SNLSintro} on another probability space. Using such approach, the well-posedness of strong solutions in the probabilistic sense (which means on the same probability space) of \eqref{eq:SNLSintro} were obtained only under some strong hypothesis. For instance, in \cite{gatarek1991continuous} such solutions were found for some specific, but physically interesting, examples. Strong solutions were also obtained in the case where $(L_k)_{k \geq 1}$ are bounded operators or $H$  is finite-dimensional since it implies that the non-linearity is locally Lipschitz. 

In a general setting, one of the main difficulty comes from the fact that if $L$ is unbounded. Hence, even for the (deterministic) Lindblad equation, the global well-posedness is unclear. To overcome this difficulty, various authors introduced a reference operator $C$ to control both the Hamiltonian and the dissipative parts. Moreover, for the stochastic equation \eqref{eq:SNLSintro}, the unboundedness of $L$ also implies that the non-linearity of the stochastic equation \eqref{eq:SNLSintro} is not locally-lipschitz in $H$.

In \cite{mora08nlss}, the authors proved the existence of weak solution in the probabilistic sense i.e. in a different probability space and uniqueness in law. However, the proof uses a Galerkin scheme which requires strong technical assumptions. In \cite{fagnola2013stochastic}, the authors succeeded in removing these assumptions by using another method of approaching the solution with Yosida approximations. In both articles, the strong existence was not proved since the pathwise uniqueness was missing. The constructed solution $\psi$ is almost surely in $L^2(0, T; D(C)) \cap C([0,T], H)$. This means that the supremum of the norm in $D(C)$ of $\psi$ may blow-up before reaching the final time $T$ if $C$ is unbounded. Hence, the sequence of stopping times defined by
\begin{equation*}
    \tau_n := \inf \{t >0, |C \psi(t)|_H > n\} \wedge T
\end{equation*}
for $n \geq 1$, may never converge to $T$ on a set of positive measure even if the following bound holds
\begin{equation*}
    \sup_{t \in [0,T]} \E |C \psi(t)|_H^2 < \infty.
\end{equation*}
Therefore, the usual truncation methods cannot be applied, which is mandatory for controlling the non-linear elements to prove the pathwise uniqueness. However, in some specific cases, one can prove that the solution lies in $C([0,T], D(C))$ almost surely. For instance it was done in \cite{gatarek1991continuous}, which allowed the authors to prove the pathwise uniqueness.

Moreover, in each of these examples, the solutions of \eqref{eq:SNLSintro} are strongly regular in the state space. Indeed, the various authors used only initial conditions in $D(C)$ for \eqref{eq:SNLSintro} and thus the solutions morally stayed in $D(C)$ at each time. In the deterministic framework, the usual initial conditions are taken in the energy space $D(C^{1/2})$, and then by using assumptions from \cite{chebotarev1998sufficient}, the solutions of \eqref{eq:lindblad} lie in $D(C^{1/2})$. This strong regularity can be noticed in the definition of $C$-regular solutions \cite[Definition 2]{mora08nlss} where projectors on $D(C)$ are apparent.

Therefore, the main objectives of this paper are first to construct solutions in the energy space $D(C^{1/2})$ for the linear equation and then for equation \eqref{eq:SNLSintro} by the usual renormalization method. This allows us to have assumptions on $C$ closer to those presented in \cite{chebotarev1998sufficient}. To do this, we prove new estimates, usual in the study of SPDE, for the equation \eqref{eq:SSE} which simplify the proof of many estimates. Then, the second and most difficult objective is to prove the pathwise uniqueness for solutions of \eqref{eq:SNLSintro}. The main difficulty is to control terms involving unbounded operators without having the continuity of the trajectories in the energy space. To get around this difficulty, we obtain a weakened monotonicity condition and a fine trajectory control, which allows us to prove pathwise uniqueness. From the uniqueness, the existence of strong solutions follows by an application of Yamada-Watanabe theorem and thus we obtain the global well-posedness. To the best of our knowledge, this paper is the first to derive energy solutions and the pathwise uniqueness of \eqref{eq:SNLSintro} with only technical assumptions in addition to the non-explosion conditions of \cite{chebotarev1998sufficient}.

\section{Preliminaries and main results}

Let $(H, \langle\cdot, \cdot\rangle_H)$ be a complex separable Hilbert space with its inner product. We consider $H$ as a real Hilbert space with the scalar product
\begin{equation*}
    (x,y)_H := \Re \langle x, y \rangle_H.
\end{equation*}
We denote the associated norm by $|\cdot|_H$. For $V$ a separable Hilbert space, with a dense embedding in $H$, we denote by $V^*$ the dual of $V$ relative to $H$. Their respective norms are denoted $|\cdot|_V$ and $|\cdot|_{V^*}$ and the duality pairing is denoted $\langle \cdot, \cdot \rangle_{V,V^*}$. We also denote by $l^2(H)$ the Hilbert space of square summable $H$-valued sequences with the following scalar product, for $u, v \in l^2(H)$
\begin{equation*}
    (u,v)_{l^2(H)} = \sum_{k=1}^{\infty}(u_k, v_k)_H,
\end{equation*}
and its associated norm $|\cdot|_{l^2(H)}$. 

For a positive operator $C : D(C) \subset H \to H$, we denote by $(\cdot, \cdot)_C$ the associated scalar product, for $x, y \in D(C)$:
\begin{equation*}
    (x,y)_C = (x,y)_H + (Cx, Cy)_H,
\end{equation*}
and $|\cdot|_C$ the norm associated with $C$. 

For a Banach space $(A, |\cdot|_A)$ and $t>0$, the space of continuous functions (resp. bounded measurable functions) on $[0,t]$ with values in $A$ is denoted:
\begin{equation*}
    C_t A := C([0,t], A) \text{ (resp. } L^{\infty}_t A := L^{\infty}(0,t;A) \text{)},
\end{equation*}
equipped with the norm:
\begin{equation*}
    |u|_{C_t A} = |u|_{L^{\infty}_t A} := \sup_{s \in [0,t]} |u(s)|_{A}.
\end{equation*}
Similarly, the space of square integrable functions on $(0,t)$ with values in $A$ is denoted:
\begin{equation*}
    L_t^2 A := L^2(0,t; A),
\end{equation*}
equipped with the norm:
\begin{equation*}
    |u|_{L^2_t A} := \left( \int_0^t|u(s)|_{A}^2\d s \right)^{1/2}.
\end{equation*}
For an operator $L : D(L) \subset H \to H$ and $X \in D(L)$, we define the following function
\begin{equation}
    \label{eq:defaverage}
    \langle L \rangle_X = (X, L X)_H.
\end{equation}
Note that, if $X$ has unit $H$-norm, then it is the average of $L$ in state $X$. Moreover, if $L$ is densely defined it can be decomposed, on the set $D(L) \cap D(L^*)$, as a sum of a symmetric operator $A$ and an anti-symmetric operator $B$, both depending on $L$. One possible choice for such operators is as follows
\begin{equation}
    \label{eq:defdecAB}
    A = \Frac{L + L^*}{2} \: \text{ and } \: B = \Frac{L - L^*}{2},
\end{equation}
defined on the adequate domains.

We consider the following stochastic nonlinear Schrodinger equation, arising in the context of quantum open systems
\begin{equation}
    \label{eq:SNLS}
    \begin{aligned}
    \d X(t) &= -\i \Hc X(t) \d t + \sum_{k=1}^{\infty}\left(L_k X(t) -   \langle L_k \rangle_{X(t)} X(t)\right) \d W^k(t)\\
    &- \frac{1}{2}\sum_{k=1}^{\infty} \left(L_k^* L_k  - 2  \langle L_k \rangle_{X(t)} L_k  +  \langle L_k \rangle_{X(t)}^2 \right) X(t) \d t.
    \end{aligned}
\end{equation}
where $\Hc :  D(\Hc) \subset H \to H$ is an unbounded self-adjoint operator, $(L_k)_{k \geq 1}$ a sequence of unbounded operators with domains $(D(L_k))_{k\geq1}$ dense subsets of $H$. The unknown function $X$ is a complex valued random field on a probability space $(\Omega, \Fc, \P)$ endowed with a standard filtration $(\Fc_t)_{t \geq 0}$. The sequence of stochastic processes $(W^k)_{k \geq 1}$ is a sequence of independent $\R$-valued $(\Fc_t)_{t\geq0}$-adapted Brownian motions on the stochastic basis $(\Omega, \Fc, (\Fc_t)_{t \geq 0}, \P)$.

Motivated by the usual proof to study equation \eqref{eq:SNLS} as explained in the introduction, we first study the following linear stochastic Schrodinger equation
\begin{equation}
    \label{eq:SSE}
    \d u(t) = \left(- \i \Hc u(t) - \Frac{1}{2}\sum_{k=1}^{\infty} L^*_k L_k u(t) \right) \d t + \sum_{k=1}^{\infty} L_k u(t) \d B^k(t),
\end{equation}
where $(B^k)_{j\geq1}$ is another sequence of independent $(\Fc_t)_{t \geq 0}$-adapted Brownian motions defined on the stochastic basis $(\Omega, \Fc, (\Fc_t)_{t \geq 0}, \P)$. We recall the definition of the effective Hamiltonian $G$ associated to this system by
\begin{equation*}
    G := -\i \Hc - \Frac{1}{2} \sum_{k=1}^{\infty} L_k^* L_k.
\end{equation*}
The domain of $G$ is defined as the elements $D(H)$ such that the operators $(L_k^*L_k)_{j \geq 1}$ is well-defined as a linear operator from $D(G)$ to $l^2(H)$. Then, similarly as in \cite{chebotarev1998sufficient} for the study of Lindblad equation \eqref{eq:lindblad}, we introduce a reference operator $C$ to control the effective Hamiltonian and the coupling operators. Therefore, we recall the key assumptions on the reference operator $C$ as stated in \cite{chebotarev1998sufficient}.
\begin{assumption}
\label{a:hyp1}
Assume that the domain $D(G)$ is dense. Moreover, there exists a linear operator $C: D(C) \subset H \rightarrow H$ is a densely defined, positive and self-adjoint operator such that
\begin{enumerate}
    \item $D(G) \subset D(C^{1/2})$ and $D(G)$ is a core for $C^{1/2}$.
    \item $D(C^{1/2}) \subset D(L_k)$ for all $k \geq 1$.
    \item There exists $K_C >0$ such that for all $x \in D(C^{3/2})$
    \begin{equation}
    \label{eq:condenergy}
        2 ( C^{1/2} x, C^{1/2} G x )_H + \sum_{k=1}^{\infty} |C^{1/2} L_k x|_H^2 \leq K_C | x |_{D(C^{1/2})}^2.
    \end{equation}
\end{enumerate}
\end{assumption}
Note that the condition \eqref{eq:condenergy} is similar to a control of the energy for the effective Hamiltonian $G$. Indeed, in \cite{chebotarev1998sufficient}, the authors proved that under Assumption \ref{a:hyp1} and if $G$ generates a contraction semigroup in $H$, then the operator $G$ is the infinitesimal generator of a strongly continuous semigroup on the Hilbert space $D(C^{1/2})$. However, to construct solutions to equation \eqref{eq:SSE} we need additional assumptions on the operator $C$. Our main set of working hypotheses is as follows.

\begin{assumption}
\label{a:hyp2}
The linear operator $C: D(C) \subset H \rightarrow H$ is such that
\begin{enumerate}
    \item $C$ satisfies Assumption \ref{a:hyp1}.
    \item There exists $\varepsilon >0$ such that :
    \begin{equation}
        C - \varepsilon I > 0.
    \end{equation}
    \item $G$ can be extended as a bounded operator from $D(C^{1/2})$ to its (topological) dual denoted $D(C^{1/2})^*$.
\end{enumerate}
\end{assumption}
Before stating the different results obtained in the present paper, we will first give some comments on the various conditions imposed on operator $C$ in Assumption \ref{a:hyp2}. Conditions $1$ is the natural one to be able to define a good reference operator since they are quite directly from \cite{chebotarev1998sufficient}. Assumption $2$ ensures that the operator $C$ is invertible. It is not a restrictive assumption since one can easily check that if $C$ is a positive operator satisfying Assumption \ref{a:hyp1}, then $C+I$ satisfies the first and second points of assumption \ref{a:hyp2}.

Condition $3$ is the main differences between Assumption \ref{a:hyp2} and the assumptions from \cite{chebotarev1998sufficient}. We explicitly demand the possibility to extend $G$ as a bounded operator from $D(C^{1/2})$ to $D(C^{1/2})^*$, whereas the condition in \cite{chebotarev1998sufficient} is that for all $u \in D(G)$ :
\begin{equation*}
    - 2 (u, G u)_H \leq |C^{1/2} u |^2_H.
\end{equation*}
However, such a condition is quite natural because it is imposed directly on the effective Hamiltonian. For example, it is imposed in \cite{chebotarev1998priori} to obtain a priori estimates for the solution of Lindblad equation \eqref{eq:lindblad}.

We can then state the global well-posedness of equation \eqref{eq:SSE}.
\begin{theorem}
\label{thm:exglobalsse}
Fix $T >0$, a stochastic basis $(\Omega, \Fc, (\Fc_t)_{t \geq 0}, \P)$ and a sequence of independent $(\Fc_t)_{t \geq 0}$- adapted Brownian motions $(B^k)_{k\geq1}$. Let $C : D(C) \subset H \to H$ be an operator satisfying Assumption \ref{a:hyp2} and let $u_0 \in L^2(\Omega, D(C^{1/2}))$ a $\Fc_0$-measurable random variable. Then, there exists a unique solution $u \in L^2(\Omega, C([0,T], H) \cap L^2([0,T], D(C^{1/2})))$ to \eqref{eq:SSE} such that $u(0) = u_0$. Moreover, for all $t \leq T$
\begin{equation}
    \E [|u(t)|_{H}^2] = \E [|u_0|^2_{H}],
\end{equation}
and there exists a constant $C(G, K_C, T) >0$ such that
\begin{equation}
    \label{eq:estimatemomentssolSSE}
    \E \left[|u|_{C_{T}H}^2 + |u|^2_{L^2_{T}V} \right] \leq C \E |u_0|^2_V.
\end{equation}
\end{theorem}

Using the terminology of \cite{mora08nlss}, we call the solution $u$ given by the previous theorem with a reference operator $C$ a strong $C^{1/2}$-solution of \eqref{eq:SSE}.

Then we are interested in the study of the normalized equation, i.e. we want to prove the well-posedness of \eqref{eq:SNLS}. First we will prove the existence of weak solutions in the probabilistic sense to \eqref{eq:SNLS}. With this in mind, the proof will be done in a similar manner as in \cite{mora08nlss} by renormalizing a strong $C^{1/2}$-solution of equation \eqref{eq:SSE}. Second, to prove the existence of strong solution, we will first prove the pathwise uniqueness for \eqref{eq:SNLS}. First, we will prove the following quite general result on indistinguishability of stochastic processes.
\begin{theorem}
\label{thm:uniqueness}
Let $(H, |\cdot|_{H})$ and $(V, |\cdot|_V)$ be two separable Hilbert spaces such that $V$ is a dense subspace of $H$. Let $G$ be a bounded linear operator from $V$ to $V^*$. Let $F$ and $J = (J_k)_{k \geq 1}$ be non-linear functions from $V$ to respec $H$ and $l^2(H)$. Assume that there exists two non-negative measurable functions $\sigma, \eta$ respectively from $V$ and $V\x V$ to $\R$, such that for all $X \in V$ 
\begin{equation}
    \label{eq:boundconduniqueness}
    |F(X)|_H + |J(X)|^2_{l^2(H)} \leq \sigma (X),
\end{equation}
and for all $X, Y \in V$
\begin{equation}
\label{eq:growthconduniq}
\begin{aligned}
2\langle X-Y, G X - G Y\rangle_{V, V^*} &+ 2 (X-Y,F(X) - F(Y))_{H} + |J(X) - J(Y)|_{l^2(H)}^2 \\
&\leq \eta(X,Y) |X-Y|_{H}^2.
\end{aligned}
\end{equation}
Assume that there exist two adapted processes $u$ and $v$, on the same stochastic basis with the same independent Brownian motions $(W^k)_{k \geq 0}$ associated to the previous basis, solutions of the equation
\begin{equation}
    \label{eq:SPDEunique}
    \d u(t) = \left( G u(t) + F(u(t)) \right) \d t + \sum_{k=1}^{\infty} J_k(u(t)) \d W^k(t),
\end{equation}
with the same initial condition. Moreover, assume that for almost surely all $\omega \in \Omega$
\begin{equation*}
    u(\cdot, \omega), v(\cdot, \omega) \in C([0,T], H) \cap L^2(0,T;V)
\end{equation*}
and
\begin{equation}
    \label{eq:finiteconduniqueness}
    \int_0^T \sigma(u(t,\omega)) + \sigma(v(t,\omega)) + \eta(u(t,\omega),v(t,\omega)) \d t < + \infty.
\end{equation}
Then, $u$ and $v$ are indistinguishable, i.e
\begin{equation*}
    \P\left(|u(t) - v(t)|_{C_TH} = 0 \right) = 1.
\end{equation*}
\end{theorem}

In the previous theorem, the condition \eqref{eq:growthconduniq} on the nonlinear functions is similar to the monotonicity condition for monotone-coercive SPDE \cite{pardoux2021stochastic}. However, there is a relaxation of the assumptions with the apparition of the function $\eta$ on the right hand side. 

To apply such a bound to the equation \eqref{eq:SNLS} and obtain the required cancellations, we will separate $L$ into a symmetrical and antisymmetrical part. To do this, we will need the following technical assumption for a positive operator $C$.

\begin{assumption}
\label{a:hyp3}
For all $k \geq 1$,  $D(C^{1/2})$ is a subset of $D(L^*_k)$. Moreover, there exist a constant $C$ such that for all $x \in V$
\begin{equation}
    \sum_{k=1}^{\infty} |L^*_k x|^2_H \leq C |x|_{D(C^{1/2})}^2.
\end{equation}
\end{assumption}

We could weaken this hypothesis by replacing $V$ by a core of $C$. Nevertheless, to simplify the algebraic manipulations we take directly $V$. This assumption is satisfied for most examples of coupling operators $(L_k)_{k \geq 1}$. However, this assumption implies that the operators $(L_k)_{k \geq1}$ are closable, which remove some non-trivial examples, like the reflection at zero on the line (see \cite[Section 5.2.]{chebotarev1998sufficient}). Then, equipped with this theorem, and applying Yamada-Watanabe Theorem in infinite dimension \cite{fahim2025yamada, theewis2025yamada}, we obtain strong solutions (in the probabilistic sense) to the nonlinear equation \eqref{eq:SNLS}.

\begin{theorem}
\label{thm:exglobalSNLS}
Fix $T >0$, a stochastic basis $(\Omega, \Fc, (\Fc_t)_{t \geq 0}, \P)$ and a sequence of independent Brownian motions $(W^k)_{k \geq 1}$ adapted to the basis. Let $C : D(C) \subset H \to H$ be an operator satisfying Assumption \ref{a:hyp2}, assume Assumption \ref{a:hyp3} holds and let $X_0 \in L^2(\Omega, D(C^{1/2}))$ a $\Fc_0$-measurable random variable such that $|X_0|_{H} =1 $ almost surely. 

Then, there exists a unique solution $X \in L^2(\Omega, C([0,T], H) \cap L^2([0,T], D(C^{1/2})))$ to \eqref{eq:SNLS} such that $X(0) = X_0$. Moreover, $\P$-almost surely, for all $t \in [0,T]$
\begin{equation}
    |X(t)|_{H}^2 = 1.
\end{equation}
\end{theorem}

To conclude this section, we will give some examples of applications to certain physical situations. First, let us start with $\Hc: D(\Hc) \subset L^2(\R) \to L^2(\R)$, the Harmonic oscillator in $L^2(\R)$, i.e
\begin{equation*}
    \Hc := \Frac{1}{2}(-\partial_x^2 + |x|^2),
\end{equation*}
and $L$ is a linear combination of the creation and annihilation operators. More precisely, recall that these operators are defined respectively by
\begin{equation*}
    a := \Frac{1}{\sqrt{2}} (x - \partial_x) \text{ and } a^{\dagger} := \Frac{1}{\sqrt{2}}(x + \partial_x),
\end{equation*}
with domain given by $D(\sqrt{\Hc})$. Then, for $\mu$ and $\nu$ complex numbers, we take the coupling operator $L$ as
\begin{equation*}
    L := \mu a + \nu a^{\dagger}.
\end{equation*}
It is easy to verify that the operator $C := H$, using commutation relations, satisfies Assumptions \ref{a:hyp1}, \ref{a:hyp2} and \ref{a:hyp3}. Therefore, Theorem \ref{thm:exglobalSNLS} can be applied and the energy space $D(C^{1/2})$ in this case is the energy space of the harmonic oscillator. 

Another direct application is the case of decoherence. Let $\Hc$ be a self-adjoint operator bounded from below by $- \lambda$ and $L$ be given by $\sqrt{\gamma}\Hc$ where $\gamma$ is a positive constant. Then, the operator $C := (\Hc + (\lambda +1))^2$ satisfies Assumptions \ref{a:hyp1}, \ref{a:hyp2} and \ref{a:hyp3}, hence Theorem \ref{thm:exglobalSNLS} applies. However, note that in this case the energy space of $C$ is the domain of $\Hc$, since $L^*L$ is of order $\Hc^2$. This gives an easy example where the energy space of the reference operator (and thus of the effective Hamiltonian) is not that of the Hamiltonian.

The paper is organized as follows. In Section \ref{sec:proofSSE}, we will prove global well-posedness of \eqref{eq:SSE} in the energy space. For that purpose, we define a scheme using the Yosida approximation and we prove the convergence in the same probability space using monotonicity arguments. Then, the proof of Theorem \ref{thm:uniqueness} is done in Section \ref{sec:proofuniq}. The key ingredient of this proof is a pathwise Gronwall lemma for a class of stochastic processes. Finally, in Section \ref{sec:proofsnls} we first prove the pathwise uniqueness of \eqref{eq:SNLS} and then the existence of weak solutions in the probabilistic setting by the usual renormalization argument. The proof of Theorem \ref{thm:exglobalSNLS} then follows from an application of Yamada Watanabe theorem. In the Appendix we recall some tools from probability theory and we present some technical calculations used in the various proofs.

\textit{Convention:} Throughout the paper, when the reference operator $C$ is fixed, the letter $K$ will refer to various constants depending only on the constant $K_C$, from Assumption \ref{a:hyp1}.

\section{Proof of Theorem \ref{thm:exglobalsse}}
\label{sec:proofSSE}

In the whole paper, we are given a self-adjoint unbounded operator $H$ and a family of unbounded operators $(L_k)_{k \in \N}$, denoted $L$. Recall that we define the unbounded operator $G$ by \eqref{eq:defG}. Furthermore, in the whole Section \ref{sec:proofSSE}, we fix an unbounded reference operator  $C : D(C) \subset H \to H$ which satisfies Assumption \ref{a:hyp2}. To simplify the notation, in the following we denote by $V$ the set $D(C^{1/2})$ and $V^*$ its dual relative to $H$. We also still denote by $G$ its extension as a bounded operator from $V$ to $V^*$. Before starting the proof, we notice a few key properties.

\begin{proposition} Let $G$ be an unbounded operator with expression \eqref{eq:defG},  $C : D(C) \subset H \to H$ be an operator satisfying Assumption \ref{a:hyp2}. Then for all $x \in V$ 
    \begin{equation}
        \label{eq:dissipGext}
        2 \langle x, G x \rangle_{V,V^*} + \sum_{j=1}^{\infty} |L_j x|_H^2 = 0,
    \end{equation}
and there exists $K(G, C) > 0$ such that for all $x \in V$:
    \begin{equation}
        \label{eq:boundLVtoH}
        \sum_{k=1}^{\infty} |L_k x|_H^2 \leq K |x|_{V}^2.
    \end{equation}
\end{proposition}

We also fix a stochastic basis $(\Omega, \Fc, (\Fc_t)_{t \geq 0}, \P)$ and a sequence of independent $(\Fc_t)_{t \geq 0}$-adapted Brownian motions $(B^k)_{k\geq 1}$. We fix $u_0 \in L^2(\Omega, V)$ a $\Fc_0$-measurable random variable.

The proof of existence is divided in two steps. The first one consists in defining the approximation scheme using Yosida approximation, as in \cite{fagnola2013stochastic} and show some new and useful estimates in a similar way as in \cite{debussche2011local}. Then, in the second part, these estimates are used to pass to the limit in the same probability space, and identification is performed as in \cite{pardoux2021stochastic}. In this second part, we use the method from the variational approach to SPDEs, as developed in \cite{pardouxt1980stochastic, zbMATH05207289}

\subsection{The Approximation scheme}

First, we define the Yosida approximation of $-C$. To this end, for $n \in \N$ we define $R_n$ by
\begin{equation}
    \label{eq:defyosida}
    R_n := n (n + C)^{-1}.
\end{equation}
From this definition, using that $C$ satisfies Assumption \ref{a:hyp2}, one can easily show that for all $n$ integer $R_n$ is a bounded symmetric operator from $V^*$ to $V$ and commutes with $C$ and $C^{1/2}$. Moreover it holds that $\| R_n\|_{\Lc(H)} \leq 1$ and therefore $\| R_n\|_{\Lc(V)} \leq 1$. Next, we define the regularisation of $G$ and $(L_k)_{k \in \N}$ used in the approximation sequence similarly as in \cite{fagnola2013stochastic}. For $n \geq 1$, we define
\begin{equation}
    \label{eq:defGnLn}
    \begin{aligned}
    &G^n := R_n G R_n, \\
    &\text{for all $k \leq n$ } L_k^n := L_k R_n \text{ and otherwise } L_k^n = 0.
    \end{aligned}
\end{equation}

We will denote by $L^n$ the sequence of operators $(L_k^n)_{k \leq n}$. From this definition, third condition of Assumption \ref{a:hyp2} and \eqref{eq:condenergy}, we easily obtain the following lemma as in \cite[Section 5.2.]{fagnola2013stochastic}.
\begin{lemma}
\label{lem:bddGnLn}
For all $n$ integer
\begin{eqnarray}
G^n \in \Lc(V^*, V) \label{eq:bddGn},\\
L^n \in \Lc(V, l^2(V)) \label{eq:bddLn}.
\end{eqnarray}
\end{lemma}
We are now equipped to define the approximation sequence. For $n \geq 1$
\begin{equation}
    \label{eq:defgalerkin}
    \begin{cases}
     \d u_n & =  G^n u_n \d t + \sum_{k=1}^n L_k^n u_n \d B^k(t) \\
     u_n(0) & = u_0
    \end{cases}
\end{equation}
Then, by Lemma \eqref{lem:bddGnLn} and elementary results about the solutions of SDE in a Hilbert space \cite[Section 6.1.1]{da2014stochastic} with bounded operators, there exists a unique solution $u_n$ in $L^2(\Omega; C([0,T], V))$, for all $n$ integer.

\begin{lemma}
\label{lem:estgalerkin}
Let $(u_n)_{n \in \N}$ be the associated sequence of solution of the Galerkin scheme defined by \eqref{eq:defgalerkin}. Then, there exist positive and finite constants $C_1(K_C, T)$, $C_2(K_C, T)$ independent of $n$, such that
\begin{eqnarray}
\sup_{0 \leq t \leq T} \E [|u_n(t)|_H^2] &\leq& \E[|u_0|_H^2], \label{eq:galerkinbdd1}\\
\sup_{0 \leq t \leq T}\E[|u_n(t)|_V^2] &\leq& C_1 \E [|u_0|_V^2], \label{eq:galerkinbdd2}\\
\E[|u_n(t)|_{C_{T} H}^2] &\leq& C_2 \E [|u_0|_V^2]. \label{eq:galerkinbdd3}
\end{eqnarray}
\end{lemma}

The first two estimates are well-known with this set of assumptions, whereas \eqref{eq:galerkinbdd3} is new for this equation (but well-known for a large class of SPDE). For the sake of clarity, we briefly present their proofs.

\begin{proof}

Estimate \eqref{eq:galerkinbdd1} comes from an application of Ito's formula to $(|u_n(t)|^2)_{t \leq T}$, the symmetry of $R_n$, \eqref{eq:dissipGext} and by finally taking the expectation.

For estimate \eqref{eq:galerkinbdd2}, we start by applying $C^{1/2}$ to the system \eqref{eq:defgalerkin}. Then, by Ito's formula applied to the functional $(| C^{1/2} u_n(t) |_H^2)_{t \leq T}$
\begin{eqnarray*}
\d | C^{1/2} u_n(t) |_H^2 &=& \left( 2 ( C^{1/2}u_n(t), C^{1/2} R_n G R_n u_n(t) )_H + \sum_{k=1}^n |C^{1/2}  L_k R_n u_n(t)|_H^2\right) \d t \\ 
&&+ \sum_{k=1}^n (C^{1/2}u_n(t), C^{1/2}P_n L_k u_n(t) )_H \d B^k(t).
\end{eqnarray*}
Therefore, by \eqref{eq:condenergy}, \eqref{eq:galerkinbdd1} and taking the expectation, it holds for all $t \leq T$
\begin{equation*}
    \E |u_n(t)|_V^2 \leq \E |u_0|_V^2 + K_C\int_0^t \E |u_n(s)|_V^2 \d s.
\end{equation*}
Estimate \eqref{eq:galerkinbdd2} follows directly by Gronwall's lemma.

For estimate \eqref{eq:galerkinbdd3}, first notice that by applying Burkholder-Davis-Gundy inegality \cite{prevot2007concise} for stochastic integral, there exists $c >0$ such that
\begin{equation*}
    \E \left[\sup_{0 \leq t \leq T} \big|\sum_{k=1}^n \int_0^t (u_n(s), L_k^n u_n(s) )_H \d W_s^k\big|\right] \leq c \E \left[ \left(\sum_{k=1}^n \int_0^{T}(u_n(s), L_k^n u_n(s) )^2_H \d s\right)^{1/2}\right].
\end{equation*}
Applying first Cauchy-Schwarz inequality on the right hand side, then Young's inequality and finally  \eqref{eq:boundLVtoH} it implies that
\begin{equation}
    \label{eq:BDGgal}
    \begin{aligned}
    & \E\left[\sup_{0\leq t\leq T}|u_n(t)|_H \left(\sum_{k=1}^n \int_0^T | L_k^n u_n(s) |^2_H \d s\right)^{1/2}\right] \\
&\leq \frac{1}{2c}\E\left[\sup_{0\leq t\leq T}|u_n(t)|^2_H\right] + \frac{c K_C}{2} \E\left[\int_0^T |R_n u_n(s) |_V^2 \d s\right],
    \end{aligned}
\end{equation}

By taking the expectation of the supremum in time in the application of Ito's formula to $(|u_n(t)|^2)_{t \leq T}$ and combining with \eqref{eq:BDGgal} it yields
\begin{equation*}
    \E\left[\sup_{0\leq t\leq T}|u_n(t)|_H^2\right] \leq \E [|u_0|_H^2] + \frac{1}{2}\E\left[\sup_{0\leq t\leq T}|u_n(t)|_H^2\right] + \frac{c^2 K_C}{2} \E\left[\int_0^{T} |u_n(s) |_V^2 \d s\right].
\end{equation*}
Therefore, estimate \eqref{eq:galerkinbdd3} follows \eqref{eq:galerkinbdd2}.

\end{proof}

\subsection{Passage to the limit}
\label{subsec:passlim}

With the previous estimates in hand, we are now ready to prove Theorem \ref{thm:exglobalsse}. We will exactly follow the final steps of the proof of \cite[Theorem 2.13.]{pardoux2021stochastic}. In the following, we will extract a sub-sequence three times by the Banach-Alaoglu theorem (which we will not distinguish from the initial sequence).

First, we observe that the approximations of the operators are adequate. More precisely, the approximations converge in the following, for all $v \in L^2( \Omega, L^2_T V)$
\begin{equation*}
    \lim_{n \to + \infty} \E \int_0^T (v(s), G^n v(s))_H \d s = \E \int_0^T (v(s), G v(s))_H \d s,
\end{equation*}
and
\begin{equation*}
    \lim_{n \to + \infty} \E \int_0^T | L^n v(s)|_{l^2(H)}^2 \d s = \E \int_0^T | L v(s)|_{l^2(H)}^2 \d s.
\end{equation*}
Second, Lemma \ref{lem:estgalerkin} asserts that the sequence $(u_n)_{n \in \N}$ is bounded in $L^2(\Omega; L^{\infty}_{T}H\cap L^2_{T}V)$, hence there exists $u$ in $L^2(\Omega, L^{\infty}_{T}H\cap L^2_{T}V)$ such that $(u_n)_{n \in \N}$ converges weak-star toward $u$ in this space. Third, there exists $\xi \in L^2(\Omega, L^2_{T}V^*)$ weak limit of $(G^n u_n)_{n \in \N}$ since $G \in \Lc(V,V^*)$, $(G u_n)_{n\in \N}$ is bounded in $L^2(\Omega \x [0,T]; V^*)$ by \eqref{eq:galerkinbdd2}. Fourth, combining \eqref{eq:boundLVtoH} and \eqref{eq:galerkinbdd2} yield that $(L^n u_n)_{n \in \N}$ is bounded in $L^2(\Omega \x [0,T]; l^2(H))$. Therefore, there exists a weak limit $\eta$ in $L^2(\Omega \x (0,T); l^2(H))$ to $(L^n u_n)_{n \in \N}$. Combining everything and taking the limit in \eqref{eq:defgalerkin}, we deduce that for all $t \in [0,T]$
\begin{equation}
    \label{eq:liminter}
    u(t) = u(0) + \int_0^t \xi(s) \d s + \int_0^t\sum_{k=1}^{\infty} \eta_k(s) \d W^k(s).
\end{equation}
Moreover, by lower semi-continuity, there exists a positive constant $C(G, K_C, T)$ such that
\begin{equation}
    \label{eq:lscestimlim}
    \E \left[|u|_{L^{\infty}_{T}H}^2 + |u|^2_{L^2_{T}V} \right] \leq C \E |u_0|^2_V.
\end{equation}
Finally by evaluating the scheme at time $T$, we found that $(u_n(T))_{n \in \N}$ is bounded in $L^2(\Omega, H)$. Therefore, another extraction can be performed such that $(u_n(T))_{n \in \N}$ converges weakly in $L^2(\Omega, H)$.

To conclude the proof, we need to first prove that $u$ has continuous random paths and second identify the limiting elements $\xi$ and $(\eta_k)_{k \geq 1}$.

For the first part, by \eqref{eq:liminter} and Proposition \ref{lem:itoformulanorm} for $\P$-almost every $\omega \in \Omega$, $u(\cdot, \omega)$ is continuous. Hence, $u$ lies in $L^2(\Omega, C_{T}H\cap L^2_{T}V)$ by \eqref{eq:lscestimlim}.

For the identification, we start by noticing a monotonicity bound, i.e. we have for all $x, y \in V$ 
\begin{equation*}
    2 \langle x - y, G x - G y \rangle_{V, V^*} + \sum_{k=1}^{\infty}|L_k x - L_k y|^2_H \leq 0,
\end{equation*}
by \eqref{eq:dissipGext} and the linearity of the operators. From this bound and the various convergences obtained, we can identify the limiting elements i.e. $\xi = G u$ and $\eta = (L_k u)_{k \in \N}$. The proof of this result is similar to the one of \cite[Lemma 2.18.]{pardoux2021stochastic}, we recalled the main elements of the proof in Appendix \ref{app:limit} for the sake of clarity.

Finally, we need to prove the pathwise uniqueness of the solution. For that, we will first prove the mean conservation of the $H$-norm of the solution.

\begin{proposition}
    \label{prop:meanpreserv}
    Fix $T >0$, a stochastic basis $(\Omega, \Fc, (\Fc_t)_{t \geq 0}, \P)$ and a sequence of independent Brownian motions $(W^k)_{k\geq1}$. Let $u_0 \in L^2(\Omega,V)$ a $\Fc_0$-measurable random variable. Assume that there exists a solution $u \in L^2(\Omega, C([0,T],H) \cap L^2(0,T;V))$ with initial condition $u_0$. Then, for all $t \leq T$
    \begin{equation}
        \label{eq:conservationl2omeganorm}
        \E |u(t)|^2_H = \E |u_0|^2_H.
    \end{equation}
\end{proposition}

\begin{proof}
 Since $u$ is a solution to \eqref{eq:SSE} lying in $ L^2(\Omega, C([0,T],H) \cap L^2(0,T;V))$, by Ito's formula to $(|u(t)|^2_H)_{t \geq 0}$, given by Proposition \ref{lem:itoformulanorm}, we infer that almost surely for all $t \leq T$
\begin{equation*}
    |u(t)|^2_H = |u_0|^2_H + 2\int_0^t \langle u(s), G u(s) \rangle_{V, V^*} \d s + 2 \int_0^t (u(s), \d M_s)_H + \langle M \rangle_t,
\end{equation*}
where $M$ is the stochastic integral part of \eqref{eq:SSE}. The result follows by taking the expectation and applying \eqref{eq:dissipGext}.
\end{proof}

\begin{corollary}
\label{cor:uniciteSSE}
Let $u^1$ and $u^2$ be two solutions to \eqref{eq:SSE} in $L^2(\Omega, C([0,T],H) \cap L^2(0,T,V))$ with initial condition $u_0$. Then $u^1$ and $u^2$ are indistinguishable.
\end{corollary}

\section{Proof of Theorem \ref{thm:uniqueness}}
\label{sec:proofuniq}
This section is devoted to the proof of Theorem \ref{thm:uniqueness}. The main tool is a pathwise version of Gronwall's lemma for stochastic processes. To prove the theorem, we will first apply this lemma to the square $H$-norm of the difference with a good choice of stopping times and then we will take the expectation to conclude the localization argument.

\begin{lemma}[Pathwise Gronwall's lemma]
\label{lem:pathwisegronw}
Let $T>0$ and $(\Omega, \Fc, (\Fc)_{t \leq T}, \P)$ be a stochastic basis. Let $\tau$ be a stopping time bounded by $T$, $\phi$ be a non-negative almost-surely integrable process and $M$ be a continuous semi-martingale. Assume that $b$ is a stochastic process with almost surely continuous random paths such that, almost surely for all $t \leq \tau$
\begin{equation}
    \label{eq:condgronw}
    b(t) \leq \int_0^t b(s) \phi(s) \d s + M(t).
\end{equation}
Then, almost surely for all $t \leq T$
\begin{equation}
    \label{eq:pathwgronw}
    e^{-A(t \wedge \tau)}b(t\wedge \tau) \leq M(0) + \int_0^{t \wedge \tau} e^{-A(s)} \d M(s),
\end{equation}
where $A(t) = \int_0^t \phi(s) \d s$.
\end{lemma}

This lemma is an intermediate version of the usual stochastic Gronwall's lemma \cite{scheutzow2013stochastic}.

\begin{proof}
In the proof, we will work with the non-decreasing of almost surely finite variation process $A$ defined by $A(t) = \int_0^t \phi(s) \d s$. First, define $\Omega_0 \subset \Omega$ the measurable set such that $A$ and $\tau$ are finite and $M$ and $b$ lie in $C([0,T];\R)$. It is clear that $\Omega_0$ is a set of $\P$-measure $1$.

Now, fix $\omega \in \Omega_0$. From \eqref{eq:condgronw}, since $A(\omega)$ has finite variation on $[0, \tau]$, an application of Gronwall's lemma holds that for all $t \leq \tau$
\begin{equation*}
    b(t) \leq M(t) + \int_0^t e^{A(t) - A(s)} M(s) \d A(s).
\end{equation*}
Define the following intermediate process, for all $t \leq T$
\begin{equation*}
    \lambda(t) := e^{-A(t)} M(t).
\end{equation*}
Hence, for all $t \leq \tau$
\begin{equation}
    \label{eq:lemgroninterm}
    e^{-A(t)} b(t) \leq \lambda (t) + \int_0^t \lambda(s) \d A(s).
\end{equation}
Thus, by a direct application of Ito's formula to $(\lambda(t))_{t \geq0}$, for all $t \leq \tau$ it holds that
\begin{equation*}
    e^{-A(t)}b(t) \leq M(0) + \int_0^{t} e^{-A(s)} \d M(s).
\end{equation*}
Therefore, the result follows since the previous inequality is pathwise.
\end{proof}

We are now ready to prove Theorem \ref{thm:uniqueness}.

\begin{proof}[Proof of Theorem \ref{thm:uniqueness}]
Let $u$ and $v$ be two solutions of equation \eqref{eq:SPDEunique} on the same stochastic basis $(\Omega, \Fc, (\Fc_t)_{t\geq 0}, \P)$ with the same initial condition. For $n \in \N$, define the following stopping times
\begin{equation*}
    \tau_n := \inf\{ t \geq 0, |u(t)|_H \vee |v(t)|_H \vee\int_0^t \eta (u(s), v(s)) +  \sigma (u(s)) + \sigma (v(s))\d s \geq n\} \wedge T.
\end{equation*}
From the almost sure continuity of the trajectories and \eqref{eq:finiteconduniqueness}, it is clear that almost surely $\tau_n$ goes to $T$ as $n$ goes to infinity. Define the following intermediate process
\begin{equation*}
    b(t) := |u(t) - v(t)|^2_H,
\end{equation*}
which is continuous almost surely, since both $u$ and $v$ are in $C([0,T];H)$ almost surely.

By an application of Ito's formula to the functional $(b(t))_{t \leq T}$, it yields
\begin{equation}
    \label{eq:itouniq}
    \begin{aligned}
    \d b(t \wedge \tau_n) &= 2 \langle u(t\wedge \tau_n) - v(t\wedge \tau_n), G u(t \wedge \tau_n) - G v(t \wedge \tau_n) \rangle_{V, V^*} \d t \\
    &+ 2 ( u(t\wedge \tau_n) - v(t\wedge \tau_n), F(u(t\wedge \tau_n)) - F(v(t\wedge \tau_n)))_H \d t\\
    &+ |J(u(t\wedge \tau_n)) - J(v(t\wedge \tau_n))|^2_{l^2(H)} \d t + \d M^n(t)
    \end{aligned}
\end{equation}
where $M^n$ is given by
\begin{equation*}
    \d M^n(t) := 2 \sum_{l=1}^{\infty} (u(t\wedge \tau_n) - v(t\wedge \tau_n), J_l(u(t\wedge \tau_n)) - J_l(v(t\wedge \tau_n)))_H \d W^l(t\wedge \tau_n),
\end{equation*}
with $M^n(0) = 0$. One can easily show that with the definition of $\tau_n$ and \eqref{eq:boundconduniqueness}
\begin{equation*}
    \E |M^n(T)|^2 < + \infty,
\end{equation*}
and thus it is a continuous martingale since it is a square integrable stochastic integral. Moreover, by applying \eqref{eq:growthconduniq} in \eqref{eq:itouniq}, we have that for all $t \leq \tau_n$
\begin{equation*}
    b(t) \leq \int_0^t b(s) \eta(u(s),v(s)) \d s + M^n(t).
\end{equation*}
Then, we can define
\begin{equation*}
    \phi(t) := \eta(u(t),v(t)),
\end{equation*}
which clearly is a non-negative process and almost-surely intergable by \eqref{eq:finiteconduniqueness}. Therefore, by applying Lemma \ref{lem:pathwisegronw}, it holds for all $t \leq T$
\begin{equation}
    \label{eq:appligron1}
    b(t \wedge \tau_n) e^{- \int_0^{t \wedge \tau_n} \eta(u(s),v(s)) \d s} \leq  \int_0^{t \wedge \tau_n} e^{-A(s)} \d M^n(s).
\end{equation}
For the right hand side of \eqref{eq:appligron1}, by Ito's isometry and Cauchy-Schwarz inequality it holds
\begin{equation}
    \E \left( \int_0^{t \wedge \tau_n} e^{-A(s)} \d M^n(s)\right)^2 \leq K n \E \int_0^{t \wedge \tau_n} \sigma(u(s)) + \sigma(v(s)) \d s \leq K n^2.
\end{equation}
This implies that the stochastic integral is, in fact, a square integral martingale starting from $0$. Therefore, we can take the expectation in \eqref{eq:appligron1}, for all $t \leq T$
\begin{equation*}
    e^{-n} \E [b(t\wedge \tau_n)] \leq \E \left[ e^{- \int_0^{t \wedge \tau_n} \eta(u(s),v(s))\d s}b(t\wedge \tau_n)  \right] \leq 0,
\end{equation*}
where we used the definition if the stopping time. Hence, it yields for all $n \geq 1$ and for all $t \leq T$
\begin{equation*}
    \E |u(t\wedge \tau_n) - v(t\wedge \tau_n)|^2_H = 0.
\end{equation*}
Then, by letting $n$ goes to infinity, we can conclude since $u$ and $v$ lie in $C([0,T]; H)$ almost surely.
\end{proof}

\section{Proof of Theorem \ref{thm:exglobalSNLS}}
\label{sec:proofsnls}

This section is devoted to the proof of Theorem \ref{thm:exglobalSNLS}. In subsection \ref{subsec:uniqsnls} we prove the pathwise uniqueness of \eqref{eq:SNLS} with the help of Theorem \ref{thm:uniqueness}. Then, we prove the existence of solutions with the right regularity in subsection \ref{subsec:existence}.

\subsection{Uniqueness}
\label{subsec:uniqsnls}

To prove the pathwise uniqueness for equation \eqref{eq:SNLS}, we will use Theorem \ref{thm:uniqueness}. According to the form of the equation given in the theorem, we define the following intermediate functions, for all $X \in V$
\begin{equation}
    F(X) := - \Frac{1}{2}\sum_{k=1}^{\infty} \left( \langle L_k \rangle_X^2 - 2 \langle L_k \rangle_X L_k \right)X,
\end{equation}
and for $k \geq 1$
\begin{equation}
    J_k(X) := \left( L_k - \langle L_k \rangle_X \right) X.
\end{equation}
Assume that there exist solutions such that, almost surely
\begin{equation*}
    X(\cdot, \omega) \in C([0,T];H) \cap L^2(0,T;V).
\end{equation*}
To apply the theorem, we therefore need to prove that conditions \eqref{eq:boundconduniqueness} and \eqref{eq:growthconduniq} hold, with the two following functions $\eta, \sigma$
\begin{equation}
    \eta(X, Y) := K( 1 + |X|^2_H + |Y|^2_H)(|X|^2_V + |Y|^2_V) \text{ and } \sigma(X) := K( 1 + |X|^4_H)|X|^2_V.
\end{equation}
According to the regularity assumptions on the solution, it is clear that those functions satisfy \eqref{eq:finiteconduniqueness}.

Before starting the computations, we give two useful results. By definition of the average value \eqref{eq:defaverage} and Cauchy-Schwarz inequality, notice that for an operator $A$ (potentially unbounded) and all $X \in D(A)$
\begin{equation}
    \label{eq:CSaverage}
    |\langle A \rangle_X| \leq |X|_H |A X|_H.
\end{equation}
Moreover, since $H$ is seen as a real Hilbert space, if $B$ is n anti-symmetric operator then for all $X$ such that the quadratic form is well-defined, it holds
\begin{equation}
    \label{eq:skewadjprop}
    (X, B X)_H = 0.
\end{equation}

First, we prove \eqref{eq:boundconduniqueness} by separating the study of the two intermediate functions. First, notice that by \eqref{eq:CSaverage}, for $k \geq 1$
\begin{equation*}
    \begin{aligned}
    |\left( \langle L_k \rangle_X^2 - 2 \langle L_k \rangle_X L_k \right)X| &\leq |\langle L_k \rangle_X|^2|X|_H + 2|\langle L_k \rangle_X| |L_k X|_H\\
    &\leq |X|_H(1+ 2 |X|_H^2) |L_k X|^2_H.
    \end{aligned}
\end{equation*}
Hence, by \eqref{eq:boundLVtoH}
\begin{equation}
    \label{eq:estimFsnls}
    |F(X)|_H \leq K |X|_H(1+ |X|_H^2) |X|^2_V.
\end{equation}

For the second part, notice similarly that
\begin{equation*}
    |J_k(X)|^2_H \leq 2 (|X|^4_H + 1) |L_k X|^2_H.
\end{equation*}
Thus, once again by \eqref{eq:boundLVtoH}
\begin{equation}
    \label{eq:estimJsnls}
    |J(X)|^2_{l^2(H)} \leq K (|X|^4_H + 1) |X|^2_V.
\end{equation}
Therefore, combining \eqref{eq:estimFsnls} and \eqref{eq:estimJsnls}, it yields
\begin{equation}
    |F(X)|_H + |J(X)|^2_{l^2(H)} \leq K (1 + |X|_H^4)|X|^2_V = \sigma(X).
\end{equation}

Second, we need to prove \eqref{eq:growthconduniq} which is the most involved one. The idea of the proof starts from the cancellations first noticed in \cite{gatarek1991continuous} when $L$ is self-adjoint. However, with our construction, we recall that the main difficulty is the absence of almost-sure continuity in $V$. For $k \geq 1$, we recall the following decomposition of $L_k$ with a symmetric operator $A_k$ and an anti-symmetric one $B_k$
\begin{equation*}
    L_k = A_k + B_k.
\end{equation*}
It follows easily from the expression of $A_k$ and $B_k$ (given by \eqref{eq:defdecAB}) and \eqref{eq:boundLVtoH} that for all $X \in V$
\begin{equation}
    \label{eq:boundAVtoH}
    |A X|^2_{l^2(H)} \leq K |X|^2_V,
\end{equation}
and similarly
\begin{equation}
    \label{eq:boundBVtoH}
    |B X|^2_{l^2(H)} \leq K |X|^2_V.
\end{equation}
Moreover, notice that, by linearity with respect to $L$ of the average and \eqref{eq:skewadjprop}
\begin{equation*}
    \langle L_k \rangle_X = \langle A_k \rangle_X.
\end{equation*}
We are now ready to prove \eqref{eq:growthconduniq}. We start by expanding each terms of the left hand side of \eqref{eq:growthconduniq} with the decompositions of $(L_k)_k$. Four types of terms will then appear: the Hamiltonian term, terms composed solely of A, terms composed solely of B and cross terms. In the following, we take $X, Y$ elements of $V$.

First, the expansion for the effective Hamiltonian, using \eqref{eq:defG} is as follows
\begin{equation*}
    \begin{aligned}
    2\langle X-Y, G X - G Y\rangle_{V, V^*} &= 2\langle X-Y, - \i H (X - Y)\rangle_{V, V^*} -  \sum_{k=1}^{\infty}\langle X-Y, A_k^2( X - Y) \rangle_{V, V^*} \\
    &- \sum_{k=1}^{\infty}\langle X-Y, B_k^* B_k (X - Y)\rangle_{V, V^*}\\
    &+ \sum_{k=1}^{\infty}\langle X-Y, [A_k,B_k] (X - Y)\rangle_{V, V^*}\\
    &= I + II_1 + III_1 + IV_1.
    \end{aligned}
\end{equation*}
Second, for the non-linear term of the drift part, we obtain
\begin{equation*}
    \begin{aligned}
    2 (X-Y,F(X) - F(Y))_{H} &= - \sum_{k=1}^{\infty} (X-Y,  \langle A_k \rangle_X^2 X -  \langle A_k \rangle_Y^2  Y )_H \\
    &+ \sum_{k=1}^{\infty} (X-Y,  2\langle A_k \rangle_X A_k X - 2  \langle A_k \rangle_Y A_k Y )_H \\
    &+ \sum_{k=1}^{\infty} (X-Y, 2 \langle A_k \rangle_X B_k X - 2 \langle A_k \rangle_Y B_k Y)_H \\
    &= II_2 + II_3 + IV_2.
    \end{aligned}
\end{equation*}
And third, for the quadratic variation part, we find
\begin{equation*}
    \begin{aligned}
    |J(X) - J(Y)|_{l^2(H)}^2 &= \sum_{k=1}^{\infty}|A_k(X-Y) - (\langle A_k \rangle_X X - \langle A_k \rangle_Y Y)|^2_H  +\sum_{k=1}^{\infty} |B_k(X-Y)|^2_H\\
    &+2 \sum_{k=1}^{\infty} (B_k(X-Y), A_k(X-Y) - (\langle A_k \rangle_X X - \langle A_k \rangle_Y Y))_H\\
    &= II_4 + III_2 + IV_3.
    \end{aligned}
\end{equation*}
We combine the terms in the following way, as describe before
\begin{eqnarray}
II &=& II_1 + II_2 + II_3 + II_4\\
III &=& III_1 + III_2\\
IV &=& IV_1 + IV_2 + IV_3.
\end{eqnarray}
Therefore, the left hand side of \eqref{eq:growthconduniq} becomes
\begin{equation}
    \begin{aligned}
    2\langle X-Y, G X - G Y\rangle_{V, V^*} + 2 (X-Y,F(X) - F(Y))_{H} + |J(X) - J(Y)|_{l^2(H)}^2 \\
    = I + II + III + IV. 
    \end{aligned}
\end{equation}

The term $I$ is clearly $0$.

The term $II$ is the one treated similarly as in \cite{gatarek1991continuous} since it only involves symmetric operators. After expanding each terms and some careful algebraic manipulations, we obtain for all $k \geq 1$
\begin{equation*}
    \begin{aligned}
    &- (X-Y, (A_k - \langle A_k \rangle_X)^2 X - (A_k - \langle A_k \rangle_Y)^2 Y )_H + |(A_k - \langle A_k \rangle_X)X - (A_k - \langle A_k \rangle_Y)Y|^2_H \\
    &= (X, Y)_H (\langle A_k \rangle_X - \langle A_k \rangle_Y)^2,
    \end{aligned}
\end{equation*}
where the details of the calculations can be found in Appendix \ref{app:techcalc}. Then, notice that by symmetry of $A_k$
\begin{equation}
    \label{eq:devdiffaverage}
    \begin{aligned}
    \langle A_k \rangle_X - \langle A_k \rangle_Y &= (X, A_k X) + (X, A_k Y) - (X, A_k Y) - (Y, A_k Y)\\
    &= (X-Y, A_k X + A_k Y)_H.
    \end{aligned}
\end{equation}
Therefore,by \eqref{eq:devdiffaverage}, \eqref{eq:CSaverage} and \eqref{eq:boundAVtoH}, it holds
\begin{equation}
    \label{eq:estimII}
    II \leq C |X|_H |Y|_H \left( |X|_V^2 + |Y|_V^2 \right) |X-Y|^2_H.
\end{equation}

For $III$, first notice that for all $k \geq 1$
\begin{equation*}
    -\langle X, B_k^* B_k X \rangle_{V, V^*} = - |B_k X|^2_H.
\end{equation*}
This easily implies that $III$ is equal to zero.

Finally, to tackle the term $IV$, the linear terms in $X$ and $Y$ are separated from the nonlinear terms using the following decomposition
\begin{equation}
    \label{eq:decompoIV}
    \begin{aligned}
    IV &= \sum_{k=1}^{\infty} -\langle X-Y, \left(A_k B_k + B_k A_k\right) X-Y\rangle_{V, V^*}\\
    &+ \sum_{k=1}^{\infty} 4(X-Y, \langle A_k \rangle_X B_k X -  \langle A_k \rangle_Y B_k Y)_H\\
    &= i + ii.
    \end{aligned}
\end{equation}
For $i$, notice that the operator in the bracket is anti-symmetric, this directly implies that $i$ is zero. The term $ii$ requires some tricks. First, notice that by anti-symmetry of $B$
\begin{equation}
    \label{eq:estiitrick1}
    4 (X-Y, \langle A_k \rangle_X B_k X -  \langle A_k \rangle_Y B_k Y)_H = 4 (X, B_k Y)_H \left(  \langle A_k \rangle_X -  \langle A_k \rangle_Y\right).
\end{equation}
Then, once again by the anti-symmetry of $B$ and \eqref{eq:skewadjprop}, it follows that
\begin{equation}
    \label{eq:estiitrick2}
    2 (X, B_k Y) = (X-Y, B_k X + B_k Y)_H.
\end{equation}
Hence, by substituting \eqref{eq:estiitrick1}, then \eqref{eq:devdiffaverage} and \eqref{eq:estiitrick2}, into the expression for $ii$, it holds
\begin{equation}
    ii \leq K \sum_{k=1}^{\infty} |B_k (X+Y)|_H |A_k(X+Y)|_H |X-Y|^2_H.
\end{equation}
Finally, by Cauchy-Schwarz inequality, \eqref{eq:boundAVtoH}, \eqref{eq:boundBVtoH} and \eqref{eq:decompoIV}, we obtain
\begin{equation}
    \label{eq:estimIV}
    IV \leq K (|X|^2_V + |Y|^2_V)|X-Y|^2_H.
\end{equation}
Therefore, combining \eqref{eq:estimII} and \eqref{eq:estimIV} with the cancellation of $I$ and $III$, it follows that
\begin{equation*}
    \begin{aligned}
    2&\langle X-Y, G X - G Y\rangle_{V, V^*} + 2 (X-Y,F(X) - F(Y))_{H} + |J(X) - J(Y)|_{l^2(H)}^2 \\
    &\leq K (1 + |X|_H^2 + |Y|_H^2) (|X|^2_V + |Y|^2_V) |X-Y|^2_H = \eta(X, Y) |X-Y|^2_H,
    \end{aligned}
\end{equation*}
which is exactly \eqref{eq:growthconduniq}. Hence, all the conditions of Theorem \ref{thm:uniqueness} are satisfied.

Therefore, by an application of Theorem \ref{thm:uniqueness}, if $u$ and $v$ are two solutions of \eqref{eq:SNLS} with the same initial condition and paths in $C([0,T];H) \cap L^2(0,T;V)$ then $u$ and $v$ are indistinguishable. Hence, pathwise uniqueness holds in $C([0,T];H) \cap L^2(0,T;V)$.

\subsection{Existence}
\label{subsec:existence}

To prove the existence of strong solutions, we will first prove the existence of solution in another stochastic basis, i.e. weak solutions in the probabilistic sense, in the same way as in \cite{mora08nlss}. Then, we will use the pathwise uniqueness and Yamada-Watanabe theorem to prove the existence of solutions in the same stochastic basis i.e. strong solutions in the probabilistic sense.

Fix $T>0$, a stochastic basis $(\Omega, \Fc, (\Fc_t)_t, \P)$ and $(B^k)_{k \in \N}$ a sequence of $(\Fc_t)_{t \geq 0}$-Brownian motions. Fix also $u_0 \in L^2(\Omega, V)$ with almost-surely unit norm in $H$. By Theorem \ref{thm:exglobalsse}, there exists a unique $u$ strong $C^{1/2}$-solution of \eqref{eq:SSE}. 

Similarly as in \cite{mora08nlss}, we aim to prove that there exists a solution for \eqref{eq:SNLS} on the probability space $(\Omega, \Fc, (\Fc_t)_{t \leq T}, \Q)$, where $\Q$ is defined by its Radon-Nikodym derivative $\d \Q := |u(T)|^2_H \d \P$. In the following, we will denote by $\E_{\P}$ (respectively $\E_{\Q}$) the expectation with respect to $\P$ (respectively $\Q$). Since $\E|u_0|^2_H$ is unitary, by Proposition \ref{prop:meanpreserv}, $\Q$ is a probability measure. Before constructing the solution, we need to prove the martingale property of the non-negative process $(|u(t)|^2_H)_{t \leq T}$.

\begin{proposition}
 Let $(u(t))_{t \leq T}$ be a strong $C^{1/2}$-solution of equation \eqref{eq:SSE} with initial condition $u_0 \in L^2(\Omega, V)$ a $\Fc_0$-measurable random variable. Then, the process $(| u(t) |^2)_{t \leq T}$ is a continuous martingale with respect to the filtration $(\Fc_t)_{t\leq T}$.
\end{proposition}

\begin{proof}
We will prove this result in two steps: first we will show that it is a local martingale and then use some properties of local martingales to conclude.

Let us define the sequence of stopping times $(\sigma_n)_{n \in \N}$ by, for all $n \in \N$:
\begin{equation}
    \sigma_n = \inf \{ t>0, \sum_{k=1}^{\infty} \int_0^t (u(s), L_k u(s) )_H^2\d s \geq n\} \wedge T.
\end{equation}
Then, an application of Ito's formula to the process $(| u(t) |^2)_{t \leq T}$ gives for $t \in [0,T]$
\begin{equation}
    \label{eq:itonormHu}
    |u(t \wedge \sigma_n)|^2 = |u_0|^2 + \sum_{k=1}^{\infty} \int_0^{t \wedge \sigma_n} 2 (u(s), L_k u(s) )_H \d B^k_s,
\end{equation}
where we used \eqref{eq:dissipGext}. Thus, since $u$ lies in $L^2(\Omega, C_TH \cap L^2_TV)$ and by \eqref{eq:boundLVtoH} it follows that the squared $H$-norm of $u$ is indeed a local martingale.

Finally, since $|u|^2_{C_TH}$ is an element of $L^1(\Omega)$, it implies that $(|u(t)|^2)_{t \leq T}$ is a real local martingale dominated by an integrable random variable. Thus it is a true martingale.
\end{proof}

Hence, since $(| u(t) |^2)_{t \leq T}$ is a continuous non-negative martingale, we obtain
\begin{equation*}
    \Q( |u(t)|^2_H > 0, \space \forall t \in [0,T]) = \Q(|u(T)|^2_H > 0) = 1.
\end{equation*}
We can then define the process $X$ on a set of $\Q$-measure $1$, for $t \in [0,T]$ :
\begin{equation}
    X(t, \omega) := \left\{ \begin{array}{rcl}
         \frac{u(t,\omega)}{|u(t, \omega)|_H} & \mbox{if}
         & \omega \in \{|u(T)|_H > 0\} \\ 0  & \mbox{if} & \omega \in \{ |u(T)|_H = 0\}.
                \end{array}\right.  
\end{equation}
From the regularity of the process $u$, it follows that $\Q$-a.s the process $X$ lies in $C_TH$ and is of unit $H$-norm for each $t \in [0,T]$. Moreover, by definition of $\Q$, we obtain
\begin{equation*}
    \E_{\Q}|X(t)|^2_V = \E_{\P}\left[\Frac{|u(t)|^2_V}{|u(t)|^2_H}|u(T)|^2_H \right] = \E_{\P}|u(t)|^2_V,
\end{equation*}
since $(| u(t) |^2)_{t \leq T}$ is a martingale. Hence, since $u \in L^2(\Omega, L^2_TV, \P)$, it yields that $X $ is in $L^2( \Omega, L^2_T V, \Q)$.

Moreover, the martingale property allows us to apply Girsanov's theorem for absolutely continuous change of measures (as stated in \cite{lenglart1977transformation}). By Lévy's Characterization of Brownian Motion and thanks to \eqref{eq:boundLVtoH} (as in \cite{mora08nlss}), under the measure $\Q$ defined by $\d \Q = |u(T)|^2 \d \P$, the sequence of processes defined by
\begin{equation}
    \label{eq:defBkGir}
    W^k(\cdot) = B^k(\cdot) - \int_0^{\cdot} 2\frac{1}{|u(s)|^2_H} (u(s), L_k u(s) )_H \d s,
\end{equation}
is a sequence of independent $(\Fc_t)_{t \leq T}$-adapted Brownian motions. Let us define the following process
\begin{equation}
    M(\cdot) := \int_0^{\cdot}2 \sum_{k=1}^{\infty}(X(s) , L_k X(s) )_H \d W^k(s).
\end{equation}
Since $X$ lies in $L^2( \Omega, L^2_TV, \Q)$ and \eqref{eq:boundLVtoH} holds, it follows that $M$ is a continuous square integrable martingale under $\Q$. Moreover, by \eqref{eq:itonormHu} and \eqref{eq:defBkGir}
\begin{equation}
    \d |u(t)|^2_H = |u(t)|^2_H \d \left(M(t) + \langle M \rangle(t) \right).
\end{equation}
Therefore, $(|u(t)|^2_H)_{t \leq T}$ is a Doléans-Dade exponential. It implies that the process $(|u(t)|^{-1}_H)_{t \leq T}$ is well defined, as a power of a Doléans-Dade exponential, and has the following Ito representation
\begin{equation}
    \d |u(t)|^{-1}_H = |u(t)|^{-1}_H \d \left(-\Frac{1}{2}M(t) - \Frac{1}{8} \langle M \rangle(t) \right).
\end{equation}
Finally, since $u$ is a solution of \eqref{eq:SSE}, we can apply the Ito formula to the process $X$, as a product of two continuous semi-martingales, which is exactly \eqref{eq:SNLS}. Therefore $X$ is a weak solution of \eqref{eq:SNLS} in the probability space $(\Omega, \Fc, (\Fc_t)_{t \geq 0}, \Q)$ with the the sequence of independent Brownian motions $(W^k)_{k \geq 1}$ adapted to the filtration $(\Fc_t)_{t \geq 0}$.

Finally, by pathwise uniqueness and continuity almost surely in $H$, we can apply an infinite dimensional version of Yamada-Watanabe theorem \cite{fahim2025yamada, theewis2025yamada}. To specify in the case of \cite{theewis2025yamada}, we apply Theorem $3.1$ with the conditions of the variational framework given in Example $2.6$. Therefore, equation \eqref{eq:SNLS} has a unique probabilistic strong solution, which concludes the proof of Theorem \ref{thm:exglobalSNLS}.

\begin{appendix}

\section{Some probabilistic reminders}
\label{app:protools}

We recall some elementary properties on the stochastic integral in a Hilbert space $H$. For an in depth study, we refer to \cite{da2014stochastic, prevot2007concise}. We start by fixing $T>0$, a stochastic basis $(\Omega, \Fc, (\Fc_t)_{t \geq 0}, \P)$ and a sequence of independent Brownian motions $(W^k)_{k \geq 1}$ adapted to the filtration $(\Fc_t)_{t \geq 0}$.

\begin{theorem}
Let $\Phi \in L^2(\Omega, L^2(0,T; l^2(H))$, then the stochastic process $\Phi \cdot W$ defined by, for all $t \leq T$
\begin{equation*}
    \Phi \cdot W(t) := \sum_{k=1}^{\infty}\int_0^t \Phi^k(s) \d W^k(s),
\end{equation*}
is a continuous square integrable $H$-valued martingale. Moreover, the Ito isometry is satisfied, i.e for all $t \leq T$
\begin{equation*}
    \E |\Phi \cdot W(t)|^2_H = \int_0^t \E |\Phi(s)|^2_{l^2(H)} \d s.
\end{equation*}
\end{theorem}

We also recall, an application of Ito's formula for the squared norm for a large class of stochastic process.

\begin{proposition}[Ito's formula for the squared norm]
    \label{lem:itoformulanorm}
     Let $u_0 \in H$, $u$ and $v$ two adapted processes with trajectories in $L^2(0,T;V)$ and $L^2(0,T,V^*)$ respectively, an adapted process $\phi$ lies in $L^1(0,T;H)$ almost surely and $X$ be a continuous $H$-valued martingale, such that:
 \begin{equation*}
     u(t) = u_0 + \int_0^t v(s) \d s + X_t + \int_0^t \phi(s) \d s
 \end{equation*}
 Then $u \in C([0,T], H)$ a.s. Moreover almost surely,  for all $t \leq T$
 \begin{equation*}
     |u(t)|^2_H = |u_0|^2_H + 2 \int_0^t \langle u(s), v(s) \rangle_{V,V^*} \d s + 2 \int_0^t (u(s),  \phi(s) \d s + \d X(s) )_H + \langle X \rangle_t. 
 \end{equation*}
\end{proposition}
In \cite[Lemma 2.14]{pardoux2021stochastic}, the proof is given when $\phi = 0$. However, in the proof, adding a continuous process with finite variation does not change anything, since it does not hinder the application of Ito's formula.

\section{Identification of the limits of the scheme}
\label{app:limit}

This section is devoted to the identification of the limiting elements of the approximation scheme. Recall that $\xi$ is the weak limiting element of $(G^n u_n)_{n \in \N}$ in $L^2( \Omega, L^2_TV^*)$ and $\eta$ of $(L^nu_n)_{n \in \N}$ in $L^2(\Omega, L^2_Tl^2(H))$. Recall also that $u$ is the limiting element of $(u_n)_{n \in \N}$ in $L^2(\Omega, L^{\infty}_TH \cap L^2_T V)$ and satisfy \eqref{eq:liminter}. We will prove  the following identities

\begin{equation}
    \label{eq:identification}
    \xi = G u \text{ and } (\eta_k)_{k\in \N} = (L_k u)_{k\in \N}. 
\end{equation}

The proof is exactly as the proof of \cite[Lemma 2.18.]{pardoux2021stochastic}, the main arguments of which we will reproduce below for the sake of clarity.

The first part of the proof consists in showing the following inequality, for all element $v$ in $L^2( \Omega \x (0,T), V)$
\begin{equation}
    \label{eq:monotnicityweaklim}
    2 \E \int_0^T \langle u(s) - v(s), \xi(s) - G v(s) \rangle_{V, V^*} \d s + \E \int_0^T |\eta(s) - L v(s)|^2_{l^2(H)} \d s \leq 0.
\end{equation}
Second, equipped with this inequality we will identify first $(\eta_k)_{k \geq 1}$ and then the drift term $\xi$.

To prove \eqref{eq:monotnicityweaklim}, we start by proving 
\begin{equation}
    \label{eq:dissiplimweak}
    \begin{aligned}
    \E &\int_0^T 2\langle u(s), \xi(s) \rangle_{V, V^*} + |\eta(s)|^2_{l^2(H)} \d s \\
    &\leq \Liminf_{n \to +\infty}  \E \int_0^T 2\langle u_n(s), G^n u_n(s) \rangle_{V, V^*} + |L^n u_n(s)|^2_{l^2(H)} \d s.
    \end{aligned}
\end{equation}
By an application of Ito's formula to $(|u(t)|^2_H)_{t \geq 0}$, the squared norm of the limiting element $u$ and taking the expectation, it holds that
\begin{equation}
    \label{eq:itolimweak}
    \E |u(T)|^2_H - \E |u_0|^2_H = 2 \E \int_0^T \langle u(s), \xi(s)\rangle_{V, V^*} \d s + \E \int_0^T |\eta(s)|^2_{l^2(H)} \d s.
\end{equation}
since $u$ satisfies \eqref{eq:liminter}. 

Moreover, since we have chose the sub-sequence such that $(u_n(T))_{n \geq 1}$ is weakly convergent in $L^2(\Omega, H)$, it is known that
\begin{equation}
    \label{eq:liminfenT}
    \E |u(T)|^2_H \leq \Liminf_{n \to +\infty} \E |u_n(T)|^2_H.
\end{equation}
Therefore \eqref{eq:dissiplimweak} holds by combining \eqref{eq:liminfenT}, \eqref{eq:itolimweak} and the expression of $(|u_n(t)|^2_H)_{t \geq 0}$ given by Ito's formula. 

Now, notice that by \eqref{eq:dissipGext}, for all $v \in L^2( \Omega \x (0,T), V)$  it holds
\begin{equation}
    \label{eq:signscheme}
    2 \E \int_0^T \langle u_n(s) - v(s), G^n u_n(s) - G^n v(s) \rangle_{V, V^*} \d s + \E \int_0^T |L^n u_n(s) - L^n v(s)|^2_{l^2(H)} \d s \leq 0.
\end{equation}
Then, the monotonicity condition \eqref{eq:monotnicityweaklim} follows from \eqref{eq:signscheme}, \eqref{eq:dissiplimweak}, operator convergences and the various weak convergences precised at the beginning in Subsection \ref{subsec:passlim}.

We are now equipped to identify the limiting elements. First, we evaluate \eqref{eq:monotnicityweaklim} for $v = u$, which implies
\begin{equation*}
    \E \int_0^T |\eta(s) - L u(s)|^2_{l^2(H)} \d s = 0,
\end{equation*}
and thus $\eta = L u$. To conclude, we will use the Minty trick \cite[Lemma 1.]{zbMATH03200902} to identify $\xi$. More precisely, we take $w \in L^2(\Omega \x (0,T); V)$, $\theta >0$ and we evaluate \eqref{eq:monotnicityweaklim} for $v = u - \theta w$. We obtain
\begin{equation*}
    \E \int_0^T \langle w(s), \xi(s) - G u(s) \rangle_{V, V^*} \d s + \theta \E \int_0^T \langle w(s), G w(s) \rangle_{V, V^*} \d s \leq 0.
\end{equation*}
Therefore, taking $\theta$ to $0$ allows us to conclude that $\xi = G u$, as required.

\section{Some technical computations}
\label{app:techcalc}

In this section, we will present some of the heavy computations done in Subsection \ref{subsec:uniqsnls}. The aim is to show that
\begin{equation*}
    \begin{aligned}
    &- (X-Y, (A_k - \langle A_k \rangle_X)^2 X - (A_k - \langle A_k \rangle_Y)^2 Y )_H + |(A_k - \langle A_k \rangle_X)X - (A_k - \langle A_k \rangle_Y)Y|^2_H \\
    &= (X, Y)_H (\langle A_k \rangle_X - \langle A_k \rangle_Y)^2.
    \end{aligned}
\end{equation*}
First, let start by defining the following intermediate operator
\begin{equation}
    \label{eq:defalphak}
    \alpha_k(X) := A_k - \langle A_k \rangle_X I_d.
\end{equation}
Note that this operator is symmetric. Then, by expanding each terms, we find that
\begin{equation}
\label{eq:techcompinter1}
    |\alpha_k(X)X - \alpha_k(Y)Y|^2_H = |\alpha_k(X)X|^2_H + |\alpha_k(Y)Y|^2_H - 2(\alpha_k(X)X, \alpha_k(Y)Y)_H.
\end{equation}
Moreover
\begin{equation}
    \label{eq:techcompinter2}
    \begin{aligned}
    &(X-Y, (\alpha_k(X))^2 X - (\alpha_k(Y))^2 Y )_H \\
    &= (X, \alpha_k(X)^2X) + (Y, \alpha_k(Y)^2Y) - (Y, \alpha_k(X)^2X)_H - (X, \alpha_k(Y)^2Y)_H\\
    &= |\alpha_k(X)X|^2_H + |\alpha_k(Y)Y|^2_H - (Y, \alpha_k(X)^2X)_H - (X, \alpha_k(Y)^2Y)_H,
    \end{aligned}
\end{equation}
where we used the symmetry of $\alpha_k(X)$. Therefore, substracting \eqref{eq:techcompinter2} to \eqref{eq:techcompinter1}, one can notice the following factorisation
\begin{equation*}
    |\alpha_k(X)X - \alpha_k(Y)Y|^2_H - (X-Y, (\alpha_k(X))^2 X - (\alpha_k(Y))^2 Y )_H = (X, (\alpha_k(X) - \alpha_k(Y))^2 Y)_H.
\end{equation*}
Then, by \eqref{eq:defalphak}, it yields
\begin{equation*}
    \alpha_k(X) - \alpha_k(Y) = (\langle A_k \rangle_Y - \langle A_k \rangle_X) I_d.
\end{equation*}
Hence, for $k \geq 1$
\begin{equation*}
    |\alpha_k(X)X - \alpha_k(Y)Y|^2_H - (X-Y, (\alpha_k(X))^2 X - (\alpha_k(Y))^2 Y )_H = (X,Y)_H  (\langle A_k \rangle_Y - \langle A_k \rangle_X)^2,
\end{equation*}
which is the required expression.

\end{appendix}

\section*{Acknowledgments}
The author would like to thank his PhD advisors, Anne de Bouard and Gaoyue Guo, for their support and feedback on the manuscript.

\addcontentsline{toc}{section}{References}
\bibliography{biblio.bib}

\end{document}